\documentclass[11pt]{amsart}
\usepackage{graphicx}
\usepackage{amssymb, amsmath}
\newtheorem{theorem}{Theorem}[section]

\theoremstyle{definition}
\newtheorem{definition}[theorem]{Definition}
\newtheorem{example}[theorem]{Example}

\theoremstyle{remark}

\begin{document}
\title{Equationally noetherian algebras and chain conditions}
\author{\sc M. Shahryari}
\thanks{}
\address{ Department of Pure Mathematics,  Faculty of Mathematical
Sciences, University of Tabriz, Tabriz, Iran }

\email{mshahryari@tabrizu.ac.ir}
\date{\today}

%%% ----------------------------------------------------------------------
\begin{abstract}
In this article, we describe the relation between the properties of being equational noetherian and ascending chain condition on ideals of an arbitrary algebra. We also give a formulation of Hilbert's basis theorem for varieties of algebras and obtain a criterion to investigate it for a given variety.
\end{abstract}

\maketitle


\vspace{2cm}

\section{Introduction}
Universal algebraic geometry is a new area of modern algebra, whose
subject is basically the study of equations over an arbitrary
algebraic structure $A$. In the classical algebraic geometry $A$ is
a field.  Many articles already published about algebraic geometry
over groups, see \cite{BMR1}, \cite{BMR2}, \cite{BMRom}, \cite{Fed},
\cite{KM}, and \cite{MR}. In an outstanding series of papers, Z.
Sela, developed algebraic geometry over free groups to give
affirmative answers for old problems of Alfred Tarski concerning
universal theory of free groups (see \cite{SEL}). Also in \cite{KMTarski},
two problems of Tarski about elementary theory of free groups are solved. Algebraic
geometry over algebraic structures is also developed for algebras
other than groups, for example there are results about algebraic
geometry over Lie algebras and monoids, see \cite{DR1}, \cite{MSH},
and \cite{SHEV}. Systematic study of universal algebraic geometry is
done in a series of articles by V. Remeslennikov, A. Myasnikov and
E. Daniyarova in \cite{DMR1}, \cite{DMR2}, \cite{DMR3}, and
\cite{DMR4}, (during this article, we cite to these papers as
DMR-series).

In this article, after a fast review of basic concepts of universal algebraic geometry, we describe the relation between the properties of being equational noetherian and ascending chain condition on ideals of an arbitrary algebra. We also give a formulation of Hilbert's basis theorem for varieties of algebras and obtain a criterion to investigate it for a given variety. For any algebraic language $\mathcal{L}$ and a fixed algebra $A$ of type $\mathcal{L}$, we define the concept of noetherian $A$-algebra and we prove that
the term algebra $T_{\mathcal{L}(A)}(X)$  is noetherian if and only
if every $A$-algebra is $A$-equationally noetherian, provided that
$A$ has a trivial subalgebra. Here $\mathcal{L}(A)$ denotes the
language obtained from $\mathcal{L}$ by adding new constant symbols
$a\in A$. We will prove this theorem in a more general setting; we
define algebraic geometry over a pre-variety of algebras and prove
that the free algebra of a pre-variety $A$-algebras is noetherian if
and only if every element of that pre-variety
 is $A$-equationally noetherian.

The reader who needs some backgrounds of universal algebra, should see books \cite{BS}, \cite{Gor}, or \cite{Mal2}. Our notations
here almost the same as in DMR-series. Many results of this work can be stated for structures over any first order language, but for the
sake of simplicity, we restrict ourself for the case of algebraic languages.

\section{Algebraic sets and coordinate algebras}
The notations in this article are taken from DMR-series. We need to review some definitions. From now on, $\mathcal{L}$ is an arbitrary algebraic
language and $A$ is a fixed algebra of type $\mathcal{L}$. The extended language will be denoted by $\mathcal{L}(A)$ and it is obtained
from $\mathcal{L}$ by adding new constant symbols $a\in A$. An algebra $B$ of type $\mathcal{L}(A)$ is called {\em $A$-algebra}, if the map $a\mapsto a^B$ is an embedding of $A$ in $B$. Note that here, $a^B$ denotes the interpretation of the constant symbol $a$ in $B$. We assume that $X=\{ x_1, \ldots, x_n\}$ is a finite set of variables. We denote the term algebra in the language $\mathcal{L}(A)$ and variables from $X$ by $T_{\mathcal{L}(A)}(X)$. An {\em equation with coefficients from $A$} is formula of the form $p(x_1, \ldots, x_n)\approx q(x_1, \ldots, x_n)$, where
$$
p(x_1, \ldots, x_n), q(x_1, \ldots, x_n)\in T_{\mathcal{L}(A)}(X).
$$
For the sake of convenience, we will denote such an equation by
$p\approx q$. The set of all such equations will be denoted by
$At_{\mathcal{L}(A)}(X)$. Any subset $S\subseteq
At_{\mathcal{L}(A)}(X)$ is called a {\em system of equations with
coefficients from $A$}. A system $S$ is called {\em consistent}, if
there exists an $A$-algebra $B$ and an element $(b_1, \ldots,
b_n)\in B^n$ such that for all equations $(p\approx q)\in S$, the
equality
$$
p^B(b_1, \ldots, b_n)=q^B(b_1, \ldots, b_n)
$$
holds. Note that, $p^B$ and $q^B$ are the corresponding term functions on $B^n$. A system of equations $S$ is called an ideal,
if it is a congruent set on the term algebra, i.e. the set
$$
\Theta_S=\{ (p, q): p\approx q\in S\}
$$
is a congruence on $T_{\mathcal{L}(A)}(X)$. Some times we denote this congruence by the same symbol $S$. For an arbitrary system of
 equations $S$, the ideal generated by $S$, is the smallest congruent set containing $S$ and it is denoted by $[S]$.

Suppose $B$ is an $A$-algebra. An element $(b_1, \ldots, b_n)\in B^n$  will be denoted by $\overline{b}$, some times. Let  $S$ be
 a system of equations with coefficients from $A$. Then the set
$$
V_B(S)=\{ \overline{b}\in B^n: \forall (p\approx q)\in S,\ p^B(\overline{b})=q^B(\overline{b})\}
$$
is called an {\em algebraic set}. It is clear that for any non-empty family $\{ S_i\}_{i\in I}$, we have
$$
V_B(\bigcup_{i\in I}S_i)=\bigcap_{i\in I}V_B(S_i).
$$
So, we define a closed set in $B^n$ to be empty set or any finite union of algebraic sets. Therefore, we obtain a topology on $B^n$,
which is called {\em Zariski topology}. For a subset $Y\in B^n$, its closure with respect to Zariski topology is denoted by $\overline{Y}$. For any set $Y$, we define
$$
Rad_B(Y)=\{ p\approx q: \forall\ \overline{b}\in Y,\ p^B(\overline{b})=q^B(\overline{b})\}.
$$
It is easy to see that $Rad_B(Y)$ is an ideal. Any ideal of this type is called a {\em radical ideal}.  The {\em coordinate algebra}
of $Y$ is  the quotient algebra
$$
\Gamma(Y)=\frac{T_{\mathcal{L}(A)}(X)}{Rad_B(Y)}.
$$
An arbitrary element of $\Gamma(Y)$ will be denoted by $[p]_Y$.

Suppose $Y\subseteq B^n$ and $p$ is a term. Define a function $p^Y:Y\to B$ by the rule
$$
p^Y(\overline{b})=p^B(b_1, \ldots,b_n).
$$
This is called a {\em term function} on $Y$. The set of all such functions will be denoted by $T(Y)$ and it is naturally an $A$-algebra.
It is easy to see that the map $[p]_Y\mapsto p^Y$ is a well-defined $A$-isomorphism. So, we have $\Gamma(Y)\cong T(Y)$.

\section{Equational Noetherian Algebras and Chain Conditions}

An algebra $B$ is said to be {\em equationally noetherian}, if for
any system of equations $S\subseteq At_{\mathcal{L}}(X)$, there
exists a finite sub-system $S_0$, such that $V_B(S)=V_B(S_0)$. More
generally, an $A$-algebra $B$ is called {\em $A$-equationally
noetherian}, if for any system of equations $S\subseteq
At_{\mathcal{L}(A)}(X)$, there exists a finite subset $S_0\subseteq
S$, such that $V_B(S)=V_B(S_0)$. It is easy to show that we can
choose $S_0$ from $[S]$, rather than $S$. It is proved that (see
\cite{DMR2}) an $A$-algebra $B$ is $A$-equationally noetherian, if
and only if every descending  chain of closed subsets in $B^n$ has
finite length, for any $n$.  Remember that a system of equations
said to be an ideal if
$$
\Theta_S=\{ (p,q): p\approx q\in S\}
$$
is a congruence on $T_{\mathcal{L}(A)}(X)$. We say that $S$ is
$A$-ideal if for any distinct $a_1, a_2\in A$, we never have
$a_1\approx a_2\in S$. Similarly, a congruence $R$ on an $A$-algebra
$B$, is called $A$-congruence, if $a_1Ra_2$, with $a_1, a_2\in A$,
implies $a_1=a_2$. An $A$-algebra is called {\em noetherian}, if it
satisfies the ascending chain condition on $A$-congruences. Note
that this implies that every $A$-congruence is finitely generated
and vise versa.

Suppose $A$ contains a trivial subalgebra. Then we claim that the
term algebra $T_{\mathcal{L}(A)}(X)$ is noetherian, if and only if
every $A$-algebra is $A$-equationally noetherian. We will prove this
assertion in a more general form, for any pre-variety of
$A$-algebras. Hence, we need to define the notions of universal
algebraic geometry with respect to a given pre-variety.

In the sequel, we assume that $A$ is an algebra containing a trivial
subalgebra. Suppose $\mathfrak{X}$ is a pre-variety of $A$-algebras.
As before, let $X$ be a finite set of variables. Suppose
$R_{\mathfrak{X}}$ is the smallest $A$-congruence with the property
$T_{\mathcal{L}(A)}(X)/R_{\mathfrak{X}}\in \mathfrak{X}$. It can be
easily shown that
$$
R_\mathfrak{X}=\{ (p\approx q)\in At_{\mathcal{L}(A)}(X):
\mathfrak{X}\models \forall x_1\ldots \forall x_n (p\approx q)\}.
$$
Let
$$
F_{\mathfrak{X}}(X)=\frac{T_{\mathcal{L}(A)}(X)}{R_{\mathfrak{X}}}.
$$
It is well-known that $F_{\mathfrak{X}}(X)$ belongs to
$\mathfrak{X}$ and it is freely generated by the set $X$. We denote
an arbitrary element of $F_{\mathfrak{X}}(X)$ by $\overline{p}$,
where $p$ is a term in $\mathcal{L}(A)$. Note that if $\mathcal{L}$
is the language of groups and $\mathfrak{X}$ is the variety of all
groups, then $F_{\mathfrak{X}}(X)=F(X)$, the free group with the
basis $X$. If $A$ is a group and $\mathfrak{X}$ is the class of all
$A$-groups, then $F_{\mathfrak{X}}(X)=A\ast F(X)$, the free product
of $A$ and the free group $F(X)$.

Suppose now, $B\in \mathfrak{X}$ and $\overline{b}\in B^n$. We know that there exists a homomorphism $\varphi:F_{\mathfrak{X}}(X)\to B$ such that
$$
\varphi(\overline{p})=p^B(b_1, \ldots, b_n).
$$
Therefore, if $\overline{p}_1=\overline{p}_2$, then $p_1^B(b_1, \ldots, b_n)=p_2^B(b_1, \ldots, b_n)$. This shows that the following definition
has no ambiguity.

\begin{definition}
An $\mathfrak{X}$-equation is an expression of the form $\overline{p}\approx \overline{q}$, where $p$ and $q$ are terms in the language
$\mathcal{L}(A)$. If $B$ is an $A$-algebra and $\overline{b}$ is an element of $B^n$, we say that $\overline{b}$ is a solution of
$\overline{p}\approx \overline{q}$, if $p^B(b_1, \ldots, b_n)=q^B(b_1, \ldots, b_n)$.
\end{definition}

Let $S$ be a system of $\mathfrak{X}$-equations. The set of all solutions of elements of $S$, will be denoted by $V_B^{\mathfrak{X}}(S)$.
The following observation shows that this is an ordinary algebraic set. Let $S^{\prime}$ be the set of all equations $p\approx q$ such that $\overline{p}\approx \overline{q}$ belongs to $S$. Then it can be easily verified that
$$
V_B^{\mathfrak{X}}(S)=V_B(S^{\prime}).
$$
Therefore, in the sequel we will denote the algebraic set $V_B^{\mathfrak{X}}(S)$ by the same notation $V_B(S)$. The Zariski topology arising from
algebraic sets relative to the pre-variety $\mathfrak{X}$ is the same as the ordinary Zariski topology.  If $Y\subseteq B^n$, we define
$$
Rad_B^{\mathfrak{X}}(Y)=\{ \overline{p}\approx \overline{q}: \forall \overline{b}\in Y\ p^B(b_1, \ldots, b_n)=q^B(b_1, \ldots, b_n)\}.
$$
The quotient algebra
$$
\Gamma_{\mathfrak{X}}(Y)=\frac{F_{\mathfrak{X}}(X)}{Rad_B^{\mathfrak{X}}(Y)}
$$
is the {\em $\mathfrak{X}$-coordinate algebra} of $Y$. Again, it is easy to see that $\Gamma_{\mathfrak{X}}(Y)\cong \Gamma(Y)$.
We are now, ready to prove our main theorem.

\begin{theorem}
Let $\mathfrak{Y}$ be a variety of algebras of type $\mathcal{L}$ and $A\in \mathfrak{Y}$ containing a trivial subalgebra. Let
$\mathfrak{X}=\mathfrak{Y}_A$ be the class of elements of $\mathfrak{Y}$ which are $A$-algebra. Then the free algebra
$F_{\mathfrak{X}}(X)$ is noetherian if and only if every $B\in \mathfrak{X}$ is $A$-equationally noetherian.
\end{theorem}

\begin{proof}
Suppose first that $F_{\mathfrak{X}}(X)$ is noetherian and $B\in \mathfrak{X}$. Let $S$ be a system of $\mathfrak{X}$-equations and
$$
\Theta_S=\{ (\overline{p}, \overline{q}):\ \overline{p}\approx \overline{q}\in S\}.
$$
Let $R$ be the congruence, generated by $\Theta_S$. If $R$ is not an $A$-congruence, then distinct elements $a_1$ and $a_2$ in $A$ do exist such
that $(a_1, a_2)=(\overline{a}_1, \overline{a}_2)\in R$. Let
$$
S^{\ast}=\{ \overline{p}\approx \overline{q}:\ (\overline{p}, \overline{q})\in R\}.
$$
It is easy to see that
$$
V_B(S)=V_B(S^{\ast})=\emptyset,
$$
so, we can assume that $S_0=\{ a_1\approx a_2\}\subseteq S^{\ast}$, and hence $V_B(S)=V_B(S_0)$. Therefore, we assume that $R$ is an $A$-congruence.
By assumption, $R$ is finitely generated, so there exists a finite subset $R_0\subseteq R$, such that $[R_0]=R$. Let $S_0$ be the set of
equations corresponding to $R_0$. Then $S_0\subseteq S^{\ast}$ and we have $S_0^{\ast}=S^{\ast}$. Hence
$$
V_B(S)=V_B(S^{\ast})=V_B(S_0^{\ast})=V_B(S_0).
$$
This proves that $B$ is $A$-equationally noetherian.

Conversely, suppose every $B\in \mathfrak{X}$ is $A$-equationally noetherian. Let $R$ be an $A$-congruence on $F_{\mathfrak{X}}(X)$. The algebra
$$
B_R=\frac{F_{\mathfrak{X}}(X)}{R}
$$ belongs to $\mathfrak{Y}$ and it is an $A$-algebra. So, $B_R\in \mathfrak{X}$. We denote an arbitrary element of $B_R$ by $\overline{t}/R$.
If $(\overline{p}, \overline{q})\in R$, then
\begin{eqnarray*}
p^{B_R}(\overline{x}_1/R, \ldots, \overline{x}_n/R)&=&\overline{p(x_1, \ldots, x_n)}/R\\
                                                   &=&\overline{q(x_1, \ldots, x_n)}/R\\
                                                   &=&q^{B_R}(\overline{x}_1/R, \ldots, \overline{x}_n/R).
\end{eqnarray*}
Hence, if we let $S_R$ to be the set of $\mathfrak{X}$-equations corresponding to $R$, then the generic point is a solution of $S_R$, i.e.
$$
(\overline{x}_1/R, \ldots, \overline{x}_n/R)\in V_{B_R}(S_R).
$$
Suppose $\overline{p}\approx \overline{q}$ is an arbitrary $\mathfrak{X}$-equation and
$$
(\overline{x}_1/R, \ldots, \overline{x}_n/R)\in V_{B_R}(\overline{p}\approx \overline{q}).
$$
Then, $(\overline{p}, \overline{q})\in R$. Keeping in mind this observation, now assume that
$$
R_1\subset R_2\subset \cdots
$$
is a proper chain of $A$-congruences in $F_{\mathfrak{X}}(X)$. For any $i$, let
$$
S_i=\{ \overline{p}\approx \overline{q}:\ (\overline{p}, \overline{q})\in R_i\}.
$$
Suppose $(\overline{p}_i, \overline{q}_i)\in R_{i+1}\setminus R_i$ and $L_i$ is the congruence generated by $R_i$ and $(\overline{p}_i,\overline{q}_i)$.
Finally, let $T_i$ be the set of $\mathfrak{X}$-equations corresponding to $L_i$. Then we have $S_i\varsubsetneq T_i\subseteq S_{i+1}$.
Let $B=\prod_i B_i$, where $B_i=F_{\mathfrak{X}}(X)/R_i$. Then clearly, $B\in \mathfrak{X}$ and so, it is $A$-equationally noetherian.
Since $A$ contains a trivial subalgebra, so we can assume that $B_i\leq B$, for all $i$. Now, by the above observation,
$$
(\overline{x}_1/R_i, \ldots, \overline{x}_n/R_i)\in V_{B_i}(S_i)\subseteq V_B(S_i),
$$
but it is not an element of $V_B(T_i)$ and consequently, it does not belong to $V_B(S_{i+1})$. This shows that the following chain of algebraic
sets is proper
$$
V_B(S_1)\supset V_B(S_2)\supset \cdots,
$$
which is a contradiction.
\end{proof}

Now, we are able to give an exact formulation of Hilbert's basis theorem. Suppose $\mathcal{L}$ is an algebraic language and $\mathfrak{Y}$ is a
variety of algebras of type $\mathcal{L}$. Let $A\in \mathfrak{Y}$ and $\mathfrak{X}=\mathfrak{Y}_A$ be the class of all elements of $\mathfrak{Y}$ which are $A$-algebra. If $A$ has maximal property on its ideals, is the algebra $F_{\mathfrak{X}}(X)$ noetherian?

\begin{example}
Let $\mathcal{L}=(0, 1, +, \times)$ be the language of unital rings and $\mathfrak{Y}$ be the variety of all commutative rings with unite element.
Let $A\in \mathfrak{Y}$ and  $\mathfrak{X}=\mathfrak{Y}_A$. If $X=\{ x_1, \ldots, x_n\}$, then $F_{\mathfrak{X}}(X)=A[x_1, \ldots, x_n]$ and hence
Hilbert's basis theorem is valid in this case.
\end{example}

\begin{example}
Let $\mathcal{L}=(e, ^{-1}, \cdot)$ be the language of groups. Let $\mathfrak{Y}$ be the variety of groups. Let $A$ be any group and
$\mathfrak{X}=\mathfrak{Y}_A$. Then $F_{\mathfrak{X}}(X)=A\ast F(X)$. We show that $F_{\mathfrak{X}}(X)$ is not noetherian even if $A$ has
maximal property on its normal subgroups (max-n). Consider the Baumslag-Solitar group
$$
B_{m, n}=\langle a, t: ta^mt^{-1}=a^n\rangle,
$$
where $m, n\geq 1$ and $m\neq n$. Then, as is proved in \cite{BMR1},
this group is not equationally noetherian. Let $B=A\ast B_{m,n}$.
Then $B$ is an $A$-group which is not $A$-equationally noetherian.
So, by the above theorem $A\ast F(X)$ is not noetherian, Hilbert's
basis theorem fails.
\end{example}

\begin{example}
Let $\mathfrak{Y}$ be the variety of abelian groups and $A\in \mathfrak{Y}$ be finitely generated. Suppose $\mathfrak{X}=\mathfrak{Y}_A$.
Then it is easy to see that $F_{\mathfrak{X}}(X)=A\times F_{ab}(X)$, where $F_{ab}(X)$ is the free abelian group generated by $X$. So,
$F_{\mathfrak{X}}(X)=A\times \mathbb{Z}^n$. As a $\mathbb{Z}$-module, clearly $A\times \mathbb{Z}^n$ is noetherian, so Hilbert's basis theorem is true for any finitely generated abelian group $A$ in the variety of abelian groups. As a result, every abelian group $B$ containing $A$ is $A$-equationally noetherian.
\end{example}

As we mentioned above, if $A\leq B$ and $B$ is not equationally noetherian, then it is also not $A$-equationally noetherian. So, let
$\mathfrak{Y}$ be a variety of algebras and $A\in \mathfrak{Y}$. Let $\mathfrak{X}=\mathfrak{Y}_A$. If there exists an element
$B\in \mathfrak{Y}$ which is not equationally noetherian, then by our theorem, $F_{\mathfrak{X}}(X)$ is not noetherian, so we never have a version of Hilbert's basis theorem for the variety $\mathfrak{Y}$.

\begin{example}
Let $\mathfrak{Y}$ be the variety of nilpotent groups of class at most $c$. If $A\in \mathfrak{Y}$ and $\mathfrak{X}=\mathfrak{Y}_A$ and
$B\in \mathfrak{Y}$ is not finitely generated, then by \cite{MR}, $B$ is not equationally noetherian and hence $F_{\mathfrak{X}}(X)$ is not noetherian.
\end{example}

\end{document}